\documentclass[graybox]{svmult}
\usepackage{mathptmx}       % selects Times Roman as basic font
\usepackage{helvet}         % selects Helvetica as sans-serif font
\usepackage{courier}        % selects Courier as typewriter font
\usepackage{type1cm}        % activate if the above 3 fonts are
\usepackage[bottom]{footmisc}% places footnotes at page bottom

\usepackage{amsmath,amssymb,amsfonts}
\usepackage{longtable}
\usepackage[centering]{geometry}

\usepackage{setspace}
\newtheorem{thm}{Theorem}
\newtheorem{ques}[thm]{Question}
\begin{document}

\title*{Some tables of right set properties in affine Weyl groups of type A}
\titlerunning{Some tables of right set properties}
\author{Leonard L. Scott and Ethan C. Zell}

\institute{Leonard L. Scott \at University of Virginia Department of Mathematics, Charlottesville VA  22903, \email{lls2l@virginia.edu}
\and Ethan C. Zell \at University of Virginia Department of Mathematics, Charlottesville VA  22903,  \email{ecz4wu@virginia.edu}}

%%\author{Leonard L. Scott}
%%\address{Department of Mathematics, University of Virginia, Charlottesville, VA  22903}
%%\email{lls2l@virginia.edu}

%%\author{Ethan C. Zell}
%%\address{Department of Mathematics, University of Virginia, Charlottesville, VA  22903}
%%\email{ecz4wu@virginia.edu}

%%%\date{, 2018}
%%%\date{\today}

%%\subjclass[2010]{Primary 20G,20C20,20C08; Secondary 17B04,17BO8,05E10}
%%\thanks{This research was supported by Simons Foundation grant 359383 and the University of Virginia. }

\maketitle \abstract{The tables of the title are a first attempt to
understand empirically the sizes of certain distinguished sets,
introduced by Hankyung Ko, of elements in affine Weyl groups.  The
distinguished sets themselves each have a largest element $w$, and
all other elements are constructible combinatorially from that
largest element. The combinatorics are given in the language of
right sets, in the sense of Kazhdan-Lusztig. Collectively, the
elements in a given distinguished set parameterize highest weights
of possible modular composition factors of the ``reduction modulo
$p$" of a $p^{th}$ root of unity irreducible characteristic $0$
quantum group module. Here  $p$ is a prime, subject to conditions
discussed below, in some cases known to be quite mild.  Thus, the
sizes of the distinguished sets in question are relevant to
estimating how much time might be saved in any future direct
approach to computing irreducible modular characters of algebraic
groups from larger irreducible characters of quantum groups.
Actually, Ko has described two methods for obtaining potentially
effective systems of such sets.  She has  proved one method to work
at least for all primes $p$ as large as the Coxeter number $h$, in a
context she indicates largely generalizes to smaller $p$.  The other
method, which produces smaller distinguished sets, is known for
primes  $p\geq h $ for which the Lusztig character formula holds,
but is currently unknown to be valid without the latter condition.
In the tables of this paper we calculate, for all $w$ indexing a
($p$-)regular highest weight in  the ($p$-)restricted parallelotope,
distinguished set  sizes for both methods,  for affine types A$_3$,
A$_4$,  and A$_5$. To keep the printed version of this paper
sufficiently small, we only use those $w$ indexing actual restricted
weights in the A$_5$  case. The sizes corresponding to the two
methods of Ko  are listed in columns (6) and (5), respectively, of
the tables.  We also make calculations in column (7) for a third,
more ``obvious" system of distinguished sets (see part (1) of
Proposition 1 below), to indicate how much of an improvement each of
the first two systems provides. Finally, all calculations have been
recently completed for affine type A$_6$, and  the restricted cases
listed in this paper as a final table. }
%%\end{abstract}

%%%%sometimes maketitle after abstract is preferred

\bigskip
\bigskip
\bigskip
\bigskip
This research was supported by Simons Foundation grant 359383 and the University of Virginia.

%\begin{frontmatter}

%\footnote{Corresponding author}

%\end{frontmatter}

\bigskip
\section{Introduction}
    In a recent paper \cite{Ko17} Hankyung Ko has introduced  two methods potentially useful for algorithmic calculations of
    irreducible modular characters of semisimple algebraic groups $G$ in characteristic $p>0$. It has been well known since the early '60's that the
    larger ``Weyl" (or ``standard") modules provide at least one setting for calculating such irreducible modular characters, using natural bilinear forms.
     See \cite{Won72} and \cite{St68}\footnote{Steinberg had noticed the form used by Wong earlier. His context was different, but the form he found could be
      used to construct Wong's form.  Steinberg worked with enveloping algebras, whereas Wong worked with Weyl modules, focusing on their irreducible heads. It
      is hard to know if Steinberg  had modular irreducible representations in mind with his form, though such representations were, of course, one of his interests.
      He was later the AMS reviewer of Wong's paper, and does not mention noticing such a potential application, only that he had previously observed the form in his
      (widely distributed) Yale lecture notes. }  Such bilinear form approaches were largely forgotten when Lusztig proposed
      a modular character formula \cite{Lus79}, shown to hold for primes $p$ large relative to the root system in \cite{AJS94}, with an
      explicit lower bound on $p$ given  in \cite{Fie12}.  However, the large sizes of negative examples in \cite{Wil17}, as well as the general
      need for results applying for smaller primes,  led Williamson and collaborators to a recent series of papers \cite{RW15}, \cite{AMRW17a}, \cite{AMRW17b}.  These results
       are conceptually quite elegant, and  already provide formulas  in all types with lower bounds on $p$  close to those originally  proposed by Lusztig
       (and apply for all primes in type $A$). However, the cost in computational complexity is hard to estimate.
       Briefly, characters of irreducible modules are recovered in \cite{RW15}  from ingredients in new formulas for characters of tilting modules. The ingredients themselves
       are obtained from considerations and calculations using $p$-canonical bases, and calculation of the latter involves a bilinear form; see \cite{TW16}.

      This background makes it reasonable to at least investigate what can be done to revive the original bilinear form approach, and the results of Ko we discuss here
       should be viewed in that context.  Before going further, we introduce some of the notation used in \cite{Ko17}. Much of it also largely follows Jantzen \cite{Ja03}.

\section{Notation and preliminaries}
       Let $G$ be a semisimple simply connected algebraic group  over an algebraically closed field $k$ of characteristic
$p>0$.  Fix a split maximal torus $T$ with associated root system $R$ and integral weight lattice $X=X(T)$. Fix also
a choice $R^+$ for the positive roots in $R$, and a corresponding set $X^+\subseteq X$ of dominant weights.  The sum of all fundamental dominant weights is denoted $\rho$.
 The ordinary finite Weyl group associated to $R$ is denoted $W_f$, and the affine Weyl group is $W_p$; it is the semidirect product of $p\mathbb{Z}R$ with $W_f$.
 (This semidirect product makes sense for $p$ any positive integer, such as $p=1$ or $p=h$, the Coxeter number. This will be useful later in our tables.)
 The ``dot action" of an element $w\in W_p$ on an element $\gamma\in X$ is given by $w\cdot\gamma=w(\gamma+\rho)-\rho$.
 In this action, reflections in hyperplanes passing through $-\rho$  and orthogonal to fundamental roots in $R$ form a fundamental set $s_1, s_2, ...$ of
 generators for $W_f$, in the sense of Coxeter groups. We choose  an additional element of $W_p$ to be the reflection $s_0$  in the
 hyperplane $\{x\mid (x+\rho,\alpha^{\vee}_0) = -p\}$ to complete the set of fundamental generators $S$ for $W_p$. (Here $\alpha_0$ is the maximum short root, and we do NOT
 follow Jantzen, who would use $+p$, instead of  $-p$, to define the hyperplane. In the terminology of \cite{Ko17}, our $S$ consists of reflections in the walls of the top antidominant alcove
  $C^-$, not the bottom dominant alcove.)  This is a good time to mention that the right set $\mathcal R(w)$ of an element $w\in W_p$  is defined to be $\{s\in S\mid ws<w\}.$ Here the inequality is in the sense of Bruhat(-Chevalley), which, in this case, just means that $ws$ is shorter than $w$. The terminology ``right descent set" is used in place of ``right set" in \cite{Ko17}, and appears also in \cite{TW16}.

       We will consider, for each dominant weight $\lambda \in X^+$,
 three finite-dimensional $G$-modules, $\Delta(\lambda)$,
$\Delta^0(\lambda)$, and $L(\lambda)$.   The third is the
irreducible $G$-module of highest weight $\lambda$, whose character
we wish to calculate. The first module is the Weyl (or ``standard")
module with highest weight $\lambda$. It may be obtained by
``reduction mod $p$"  (base change to $k$) from a module for a
$\mathbb Z$-form of the enveloping algebra for the complex
semisimple Lie algebra associated to $R$. This ``integral" module is
irreducible upon base change to $\mathbb Q$ or $\mathbb C$. It
follows that the character of $\Delta(\lambda)$ is known, given by
Weyl's character formula.  With somewhat more care with
coefficients, a similar process, using a quantum enveloping algebra
at a root of unity, rather than a Lie algebra enveloping algebra,
can be used to construct $\Delta^0(\lambda)$. See the discussion in
\cite{Ko17}, or the more available \cite{CPS09}, which uses the
notation $\Delta^{red}(\lambda)$. To summarize, $\Delta^0(\lambda)$
is the ``reduction mod $p$" of an irreducible quantum enveloping
algebra module at a root of unity (a $p^{th}$ root of unity, for $p$
odd).  For all but a  few primes, its character  is known to be
given by Lusztig's character formula in simply-laced types of all
ranks (for all primes in type A), and in all types when $p>h$; see
\cite[p. 273]{Tan04}. The module $L(\lambda)$ is a homomorphic image
(and the head) of  $\Delta^0(\lambda)$, and the latter is a
homomorphic image of $\Delta(\lambda)$.

It is sufficient, to motivate our tables, to work with the $p$-regular case.  This also simplifies notation, since, when $p\geq h$,  most questions about modules  for $G$ reduce to those for modules with highest weight in the orbit $W_p\cdot -2\rho$. Here $-2\rho\in C^-$ and, letting $w_0$ denote the longest word in the ordinary Weyl group $W_f$, we have $0=w_0\cdot -2\rho$. The weight $0$ is the highest weight of the $1$-dimensional trivial module $L(0)$. More generally, the dominant weights $\lambda$ in the orbit $W_p\cdot -2\rho$ are precisely the elements $w_0y\cdot -2\rho$ where $y\in W_p$ satisfies
$\ell(w_0y)=\ell(w_0)+\ell(y)$.
That is, $y$ is the shortest  element in its right  coset $W_fy$; also $w_0y$ is the longest element in its (the same) right coset of $W_f$.  Equivalently, $y$ is the unique minimal element in its right coset (using the Bruhat order),  and $w_0y$ is the unique maximal element in its right coset. Ko denotes the set of all elements $w_0y$ of $W_p$ maximal in their right coset of $W_f$ by $W^+$ (with no subscript $p$, perhaps to emphasize the independence of the definition on alcove geometry). Finally, if $w\in W^+$ then $w\cdot -2\rho$ is dominant, and we have already observed a converse.  In this way we have a natural 1-1  correspondence between dominant weights  in $W_p\cdot -2\rho$ and elements of $W^+$.

\section{The results  motivating  the tables, and a question}

       The following result in the $p$-regular case restates \cite[Prop. 4.3]{Ko17} in conclusion (2),  and combines it with a consequence (1) of \cite[Prop. A.4]{Ko17}. An analog of  \cite[Prop. 4.3]{Ko17} in the $p$-singular case, presumably allowing smaller primes, is also sketched in \cite[Rem. 4.4]{Ko17}.

\begin{proposition} Suppose $\mu, \lambda$ are dominant weights with $L(\mu)$ a composition factor of $\Delta^0(\lambda)$ and with $\lambda\in W_p\cdot-2\rho$. Then also $\mu\in W_p\cdot-2\rho$ $\emph{(well known)}$, so that there are unique elements $v,w\in W^+$ with $\mu=v\cdot-2\rho$ and $\lambda=w\cdot-2\rho$. Moreover
(1) $v\leq w$, and (2) $\mathcal R(w)\subseteq \mathcal R(v)$.

\end{proposition}
\begin{proof} The assertion that $\lambda$ and $\mu$ belong to the same ``dot" orbit is the well-known linkage principle. Applying strong linkage in the case of  Weyl modules, and using the fact that $\Delta^0(\lambda)$ is a homomorphic image of the Weyl module  $\Delta(\lambda)$, shows $\mu$ is strongly linked to $\lambda$. This strong linkage implies conclusion  (1) by using  \cite[Prop. A.4]{Ko17}. (Or use  \cite[Prop. 9.1]{PS11}.)  Conclusion (2) of the proposition follows directly from
\cite[Prop. 4.3]{Ko17}, which gives (2) after starting from a similar hypothesis in the more general case of regular dominant weights.  (They are always parameterized as belonging to $W^+\cdot \nu$ for some element $\nu$ in the top antidominant alcove $C^-$.)
\end{proof}

The following is a variation, in the $p$-regular case, on a question raised by Ko at the beginning of \cite[\S5]{Ko17}. She also gives a version which applies in the $p$-singular case.

\begin{ques} Suppose $p\geq h$, and let $ v, w \in W^+$ be such that $L(v\cdot -2\rho)$ is a composition factor of $\Delta^0(w\cdot -2\rho)$   Is it then true that  $\mathcal R(w) = \mathcal R(v)$?
\end{ques}

The arguments in the proof of \cite[Prop. 5.8]{Ko17}, in particular \cite[Lem. 5.9]{Ko17}, show that the question above reduces to the case of those $w\in W^+$ for which $w\cdot -2\rho$ is $p$-restricted.  No form of the ``Lusztig conjecture" is needed for this reduction.
\vspace{2mm}
 However, as Ko in effect observes, the question above has a positive answer whenever $\Delta^0(w\cdot -2\rho)$ is irreducible in all cases for which $w\in W^+$ and $w\cdot -2\rho$  is $p$-restricted.  This is equivalent to the (Kato form) of the Lusztig conjecture, when $p>h$, or in affine  type A (and others) when $p\geq h$. (Kato's version, for a given $p\geq h,$ says that the Lusztig character formula holds for any irreducible module $L(w\cdot -2\rho)$ with $w\cdot -2\rho$ $p$-restricted. In the quantum case, the character formula is known to hold for all dominant $w\cdot -2\rho,$ under the given, slightly different,
  requirements on $p.$ See \cite[p. 273]{Tan04}.)
\vspace{2mm}

$\emph{However, this Ko question potentially has positive answers}$
even when the Lusztig conjecture does not. The tables we present
suggest to us  that computation time (for characters of irreducible
$G$-modules) could be significantly reduced when the question has a
positive answer. The weaker, but proved, assertion (2) in
Proposition 1  provides a smaller reduction, but  could also be
useful.

\section{The tables}

%\newpage

Assume $p\geq h$. The tables deal with those $w\in\ W^+$ for which $w\leq u$ in the
Bruhat order for some $u\in W^+$ with the weight $u\cdot -2\rho$
$p$-restricted. We may rechoose  $u\cdot -2\rho$ in the same $p$ alcove as $(p-2)\rho$. Consequently, the $w$'s and $u$'s have a common unique maximal element $w_{max}$.
The $w\leq w_{max}$ in $W^+$ may all be written uniquely as
a product $w=w_0y$ with $y\in W:=W_p$ and lengths additive. The
elements $y$ are listed, as a reduced product of fundamental
reflections, in column (1) of the tables below. The element $w_0$
is listed above each table, and the length of $w$ recorded in
column (4).  Column (3) gives $w\cdot -2\rho$ as a vector of
coefficients at fundamental weights (``$\omega$-coefficient"), in
the generic case $p=h$. Column (2) gives the same element  expressed as a vector (``$\varepsilon$-vector") in
terms of a standard basis of the standard permutation module for
$W_f$, a symmetric group in each case considered (currently affine
types A$_n$ for $n=3,4,5$). The main content lies in columns (5),
(6), and  (7), discussed informally in the abstract. Column (7)
counts those
 $v\in W^+$ with $v\leq w$ in the Bruhat order; column (6) counts those among these $v$ which also satisfy
  $\mathcal R(w)\subseteq \mathcal R(v)$. Finally, column (5) counts those among these $v$ which satisfy
  $\mathcal R(w) = \mathcal R(v)$.

\vspace{4mm}

%%\end{document
%%experimental end
\newpage

\begin{center}

Recall that each $w=w_0y$, where $w_0$ is $s_1s_2s_3s_1s_2s_1$ for type $A_3$.

%affine S_3 which in the paper's terminology is type A_3
\vspace{2mm}

\begin{tabular}{|l|l|l|l|l|l|l|} \hline
$y \in W$ & $w \cdot (-2\rho)$ & $\omega$-coefficients & $l(w)$ & (5) & (6) & (7)\\ \hline \hline
$s_{0}s_{3}s_{1}s_{2}$ & $\left(3,\,1,\,-1,\,-3\right)$ & $\left(2,\,2,\,2\right)$ & $10$ & $1$ & $5$ & $8$ \\ \hline
$s_{0}s_{1}s_{2}$ & $\left(3,\,-1,\,-1,\,-1\right)$ & $\left(4,\,0,\,0\right)$ & $9$ & $1$ & $1$ & $4$ \\ \hline
$s_{0}s_{3}s_{1}$ & $\left(2,\,1,\,-1,\,-2\right)$ & $\left(1,\,2,\,1\right)$ & $9$ & $1$ & $2$ & $5$ \\ \hline
$s_{0}s_{3}s_{2}$ & $\left(1,\,1,\,1,\,-3\right)$ & $\left(0,\,0,\,4\right)$ & $9$ & $1$ & $1$ & $4$ \\ \hline
$s_{0}s_{1}$ & $\left(2,\,0,\,-1,\,-1\right)$ & $\left(2,\,1,\,0\right)$ & $8$ & $1$ & $1$ & $3$ \\ \hline
$s_{0}s_{3}$ & $\left(1,\,1,\,0,\,-2\right)$ & $\left(0,\,1,\,2\right)$ & $8$ & $1$ & $1$ & $3$ \\ \hline
$s_{0}$ & $\left(1,\,0,\,0,\,-1\right)$ & $\left(1,\,0,\,1\right)$ & $7$ & $1$ & $1$ & $2$ \\ \hline
$1$ & $\left(0,\,0,\,0,\,0\right)$ & $\left(0,\,0,\,0\right)$ & $6$ & $1$ & $1$ & $1$ \\ \hline
\end{tabular}

\vspace{4mm}

For type $A_4$, $w_0$ is $s_1s_2s_3s_4s_1s_2s_3s_1s_2s_1$.

\vspace{2mm}

\begin{longtable}
{|l|l|l|l|l|l|l|} \hline
$y \in W$ & $w \cdot (-2\rho)$ & $\omega$-coefficients & $l(w)$ & (5) & (6) & (7)\\ \hline \hline
$s_{0}s_{4}s_{1}s_{2}s_{3}s_{0}s_{4}s_{2}s_{1}s_{0}$ & $\left(6,\,3,\,0,\,-3,\,-6\right)$ & $\left(3,\,3,\,3,\,3\right)$ & $20$ & $1$ & $31$ & $52$ \\ \hline
$s_{0}s_{1}s_{2}s_{3}s_{4}s_{3}s_{2}s_{1}s_{0}$ & $\left(6,\,0,\,0,\,0,\,-6\right)$ & $\left(6,\,0,\,0,\,6\right)$ & $19$ & $5$ & $7$ & $34$ \\ \hline
$s_{0}s_{4}s_{1}s_{2}s_{3}s_{0}s_{4}s_{1}s_{0}$ & $\left(6,\,3,\,-3,\,-3,\,-3\right)$ & $\left(3,\,6,\,0,\,0\right)$ & $19$ & $3$ & $4$ & $26$ \\ \hline
$s_{0}s_{4}s_{1}s_{2}s_{3}s_{0}s_{4}s_{2}s_{0}$ & $\left(6,\,1,\,0,\,-3,\,-4\right)$ & $\left(5,\,1,\,3,\,1\right)$ & $19$ & $5$ & $6$ & $32$ \\ \hline
$s_{0}s_{4}s_{1}s_{2}s_{3}s_{0}s_{4}s_{2}s_{1}$ & $\left(5,\,3,\,0,\,-3,\,-5\right)$ & $\left(2,\,3,\,3,\,2\right)$ & $19$ & $2$ & $15$ & $38$ \\ \hline
$s_{0}s_{4}s_{1}s_{2}s_{3}s_{2}s_{0}s_{1}s_{0}$ & $\left(4,\,3,\,0,\,-1,\,-6\right)$ & $\left(1,\,3,\,1,\,5\right)$ & $19$ & $5$ & $6$ & $32$ \\ \hline
$s_{0}s_{4}s_{3}s_{1}s_{0}s_{4}s_{2}s_{1}s_{0}$ & $\left(3,\,3,\,3,\,-3,\,-6\right)$ & $\left(0,\,0,\,6,\,3\right)$ & $19$ & $3$ & $4$ & $26$ \\
\hline
\end{longtable}

\newgeometry{margin=1in,top=1in}

For type $A_5,$ $w_0$ is $s_1s_2s_3s_4s_5s_1s_2s_3s_4s_1s_2s_3s_1s_2s_1$. Below, the table includes only the entries with restricted weights:

\vspace{2mm}

\begin{longtable}{|l|l|l|l|l|l|l|} \hline
$y \in W$ & $w \cdot -2
\rho$ & $\omega$-coefficients & $l(w)$ & (5) & (6) & (7) \\ \hline \hline
$s_{0}s_{5}s_{4}s_{1}s_{0}s_{5}s_{2}s_{3}s_{4}s_{2}s_{1}s_{0}s_{5}s_{2}s_{3}s_{4}s_{0}s_{1}s_{2}s_{3}$ & $\left(10,\,6,\,2,\,-2,\,-6,\,-10\right)$ & $\left(4,\,4,\,4,\,4,\,4\right)$ & $35$ & $1$ & $279$ & $478$ \\ \hline
$s_{0}s_{5}s_{4}s_{1}s_{0}s_{5}s_{2}s_{3}s_{4}s_{2}s_{1}s_{0}s_{5}s_{2}s_{3}s_{4}s_{0}s_{1}s_{2}$ & $\left(9,\,6,\,2,\,-2,\,-6,\,-9\right)$ & $\left(3,\,4,\,4,\,4,\,3\right)$ & $34$ & $2$ & $138$ & $381$ \\ \hline
$s_{0}s_{5}s_{4}s_{1}s_{0}s_{5}s_{2}s_{3}s_{4}s_{2}s_{1}s_{0}s_{5}s_{2}s_{3}s_{0}s_{1}s_{2}$ & $\left(9,\,5,\,2,\,-2,\,-6,\,-8\right)$ & $\left(4,\,3,\,4,\,4,\,2\right)$ & $33$ & $5$ & $108$ & $301$ \\ \hline
$s_{0}s_{5}s_{4}s_{1}s_{0}s_{5}s_{2}s_{3}s_{4}s_{2}s_{1}s_{0}s_{5}s_{2}s_{3}s_{4}s_{0}s_{1}$ & $\left(8,\,6,\,2,\,-2,\,-5,\,-9\right)$ & $\left(2,\,4,\,4,\,3,\,4\right)$ & $33$ & $5$ & $108$ & $301$ \\ \hline
$s_{0}s_{5}s_{1}s_{2}s_{3}s_{4}s_{0}s_{5}s_{3}s_{4}s_{1}s_{2}s_{3}s_{0}s_{1}s_{2}s_{0}$ & $\left(9,\,5,\,1,\,-2,\,-6,\,-7\right)$ & $\left(4,\,4,\,3,\,4,\,1\right)$ & $32$ & $3$ & $83$ & $234$ \\ \hline
$s_{0}s_{5}s_{4}s_{1}s_{0}s_{5}s_{2}s_{3}s_{4}s_{2}s_{1}s_{0}s_{5}s_{2}s_{3}s_{0}s_{1}$ & $\left(8,\,5,\,2,\,-2,\,-5,\,-8\right)$ & $\left(3,\,3,\,4,\,3,\,3\right)$ & $32$ & $20$ & $55$ & $244$ \\ \hline
$s_{0}s_{5}s_{4}s_{1}s_{0}s_{5}s_{2}s_{3}s_{4}s_{2}s_{1}s_{0}s_{5}s_{2}s_{3}s_{4}s_{0}$ & $\left(7,\,6,\,2,\,-1,\,-5,\,-9\right)$ & $\left(1,\,4,\,3,\,4,\,4\right)$ & $32$ & $3$ & $83$ & $234$ \\
\hline
\end{longtable}

\end{center}

\section{Some remarks on construction of the tables}

In this section, we outline the algorithm for the construction of
the previous tables. For all programs, we utilized Sage, a
well-known programming language akin to Python, but with many more
mathematical libraries. While each program is made up of many helper
functions, we utilized two main programs: one to isolate the
``maximal" element $w\in W_p$ (defined, when $p=h$, as the element
$w$ such that $w\cdot -2\rho = (p-2)\rho$)\footnote{For $p\geq h$,
the maximal element $w$ may be defined by the condition that $w\cdot
-2\rho$ belongs to the same $p$-alcove as $(p-2)\rho$. Exact
equality need not hold when $p>h$. But $p$-alcoves contain only one
integral weight when $p=h$. } and a second to analyze the properties
of $\mathcal R(w)$, as well as the right sets of those elements in
$W^+$ less than or equal to $w$ in the Bruhat order.

Before we began actual program implementation, we constructed helper
programs which serve the following purposes: create the standard
$\varepsilon$-vectors (which are used in the calculation of the
roots and $\rho$), create $\rho$, create the roots $\alpha$, compute
the dot action of a group element on a vector, and calculate the
``height" of a weight, a metric used to tell the distance from the
maximal element. Each of the roots $\alpha$---with the exception of
$\alpha_0$---can be computed quickly using the formula
$\alpha_i=\varepsilon_i-\varepsilon_{i+1}.$ Since $\alpha_0$ is the
sum of all other $\alpha_i,$ we have that $\alpha_0$ is the
difference of the first and last $\varepsilon-$vector. Using these
roots, we may compute $\rho$ as half of the sum of all the positive
roots. Finally, we define the height of a weight as in \cite[p.
56]{CPS92}; that is, the height of a weight is defined in a way
equivalent to the sum of its coefficients when written as a sum of
fundamental roots. For a weight $\nu$ then, this is $( \nu,
\rho^\vee )$ (which for type A is just $(\nu,\rho)).$ Notice that
the functions all depend on an initial input concerning the size of
the generating set of the affine Weyl group, e.g. type $A_3,$ $A_4,$
or larger. To construct $w_0,$ the longest word in the ordinary Weyl
group, there are finitely many elements, so a smaller program was
written which computes the element of maximum length. This
calculation was carried out beforehand, and the results were
recorded. For efficacy, these results were then hard-coded into the
two main programs.

In the first program's implementation, we begin by generating all of the above variables, including the element $w_0$ and the affine Weyl group itself (which Sage allows with the function WeylGroup()). The program then proceeds as a variation of the typical minimum function, where the comparator is the height metric; that is, we look for the element $w$ such that $w\cdot -2\rho$ has height closest to $(p-2)\rho.$ Say $zw_0$ is the maximal element $w.$ Repurposing the language of \cite[Lem. 3.12.5]{CPS92}, we have that $l(z)=-l(w_0)+2(\rho,\rho^\vee)$ (where the latter term is $(\rho, \rho)$ for type A). Since the maximal element is dominant, we have that the maximal element has length $2( \rho, \rho ).$ We then subtract $l(w_0),$ and generate all elements with length $l(z)$ using the Sage function W.elements\_of\_length(). We subtract this value and generate the smaller elements for three reasons. First, we desire our output in the form $w_0y$. Second, there are substantially fewer elements of smaller length, which decreases computation time. Finally, we may restrict our search only to the dominant weights and in doing so, we must check whether $l(w_0y)=l(w_0)+l(y).$ Specifically, we reapply $w_0$ on the left and check the aforementioned condition. We only continue with these elements. With the proper elements generated, we may now proceed as a typical minimum program, first calculating the height of each element, comparing it to $(p-2)\rho,$ and keeping the closest element stored in a local variable. Once the list of elements is exhausted, the result yields the maximal element in simple reflection notation.

For the second program, we require additional helper functions, one which calculates the right set of a given element, and a second which converts $\varepsilon$-notation of vectors into $\omega$-notation for the tables. Since tables in Sage are constructed from lists, we also require the relevant outputs be placed into ordered lists. On the other hand, the right sets are stored as the values of a dictionary with the corresponding element as the key. We begin the second program by generating the standard roots, the Weyl group, and the relevant vectors, in a similar fashion to the first program. Then, we generate all elements of length 1 less than that of the maximal element, collecting only those elements both smaller in the Bruhat order and associated with dominant weights. After exhausting all elements of this length, we decrease the length and repeat this process until we reach $w_0.$ To calculate the right sets of these elements, we employ apply\_simple\_reflection\_right() for each of the elements in the generating set given by W.generators(), and simply check if there is a reduction in length on the element in question. We then keep the output stored as a set object in the dictionary. At this stage, we have an ordered list of elements less than or equal to the maximal element in the Bruhat order, a list of corresponding vectors in $\varepsilon$-notation, a list of corresponding coefficients in $\omega$-notation, and a dictionary of right sets.

For the final calculations, we iterate through the list of desired elements, comparing each element to those with higher indices. Through this comparison, we make three separate counts:

\begin{enumerate}
    \item The number of elements $v\in W^+$ following the initial element $w$ with $v\leq w$ in the Bruhat order (and which are already known to have dominant weights). This count is column (7) in the table.
    \item The number of elements from the first count with $\mathcal R(w)\subseteq \mathcal R(v)$. This count is column (6) in the table.
    \item The number of elements from the first count with $\mathcal R(w) = \mathcal R(v)$. This count is column (5) in the table.
\end{enumerate}

Since each count is specific to each element, we create ordered lists corresponding to the order of the element list. We then store these lists as a table and call the latex() function on the table object to get Sage output suitable for \LaTeX.

Of course, since the size of the ordinary Weyl group grows factorially, the number of elements of interest in the affine case grows quickly, increasing the computational burden. While we were able to run all programs up to and including type $A_4$ on a typical laptop, the calculations were usually interrupted for larger inputs. Therefore, we turned to Rivanna, an advanced computing cluster with significantly more computing power. Using Rivanna, we were able to acquire the remaining tables.
\vspace{2mm}

With the same approach, we have more recently been able to make all
the calculations for affine $A_6$.  The table below gives all  the
restricted weight cases (all restricted weights $w\cdot -2\rho$
with $w=w_0y$, $l(w)=l(w_0)+l(y)$).  Products of fundamental
reflections have been abbreviated to the sequence of their
subscripts, to save space.  For similar space considerations, the
column with $\varepsilon$-notation for $w\cdot -2\rho$ has been
replaced with an (unrelated but useful) column giving the right set
$\mathcal R(w)$ of $w$, as a set of indices of fundamental
reflections.  The full affine $A_6$ table, displaying similar rows for all $5260$ dominant weights
$w\cdot -2\rho$ in the restricted parallelotope, requires $93$ pages to display, but is available upon request.

\vspace{4mm}
%\newpage

\begin{center}

    Recall that $w=w_0y$ with $w_0=s_1s_2s_3s_4s_5s_6s_1s_2s_3s_4s_5s_1s_2s_3s_4s_1s_2s_3s_1s_2s_1.$

\end{center}

\begin{longtable}{|l|l|l|l|l|l|l|} \hline
$y \in W$ & $\mathcal R(w)$ & $\omega-$notation & $l(w_0y)$ & (5) & (6) & (7) \\ \hline \hline

$ 06510623451062345210612345210643210$ & $  \{0\}$ & $   \left(5,\,5,\,5,\,5,\,5,\,5\right) $ & 56   & 1  &  3019  & 5260 \\ \hline
$ {0}{6}{5}{1}{0}{6}{2}{3}{4}{5}{1}{0}{6}{2}{3}{4}{5}{2}{1}{0}{6}{1}{2}{3}{4}{5}{2}{1}{0}{6}{4}{3}{2}{1} $ & $  \{6, 1\}$ & $   \left(4,\,5,\,5,\,5,\,5,\,4\right) $ & 55   & 3  &  1524  &  4431  \\ \hline
$ {0}{6}{5}{1}{0}{6}{2}{3}{4}{5}{1}{0}{6}{2}{3}{4}{5}{2}{1}{0}{6}{1}{2}{3}{4}{5}{4}{2}{3}{1}{2}{0}{1} $ & $  \{1, 5\}$ & $   \left(3,\,5,\,5,\,5,\,4,\,5\right) $ & 54   & 10  &  1277  &  3752  \\ \hline
$ {0}{6}{5}{1}{0}{6}{2}{3}{4}{5}{1}{0}{6}{2}{3}{4}{5}{2}{1}{0}{6}{1}{2}{3}{4}{5}{2}{1}{0}{6}{4}{3}{2} $ & $  \{6, 2\}$ & $   \left(5,\,4,\,5,\,5,\,5,\,3\right) $ & 54   & 10  &  1277  &  3752  \\ \hline
$ {0}{6}{5}{1}{0}{6}{2}{3}{4}{5}{1}{0}{6}{2}{3}{4}{5}{2}{1}{0}{6}{5}{1}{2}{3}{4}{2}{3}{1}{2}{0}{1} $ & $  \{1, 4\}$ & $   \left(2,\,5,\,5,\,4,\,5,\,5\right) $ & 53   & 10  &  1078  &  3159  \\ \hline
$ {0}{6}{5}{1}{0}{6}{2}{3}{4}{5}{1}{0}{6}{2}{3}{4}{5}{2}{1}{0}{6}{1}{2}{3}{4}{5}{4}{2}{3}{1}{2}{0} $ & $  \{2, 5, 0\}$ & $   \left(4,\,4,\,5,\,5,\,4,\,4\right) $ & 53   & 75  &  658  &  3237  \\ \hline
$ {0}{6}{5}{1}{0}{6}{2}{3}{4}{5}{1}{0}{6}{2}{3}{4}{5}{2}{1}{0}{6}{1}{2}{3}{4}{5}{2}{1}{0}{6}{4}{3} $ & $  \{6, 3\}$ & $   \left(5,\,5,\,4,\,5,\,5,\,2\right) $ & 53   & 10  &  1078  &  3159  \\
\hline

\end{longtable}
\section{Concluding remarks}

Though not listed in the tables, all the sets whose size is given in columns (5) and (6) above can be computed by our programs.  A next step, which we hope to carry out, is to use this information, guided by the tables, to compute characters of $L(w\cdot -2\rho)$ in new cases. We would like to take this opportunity to thank the University of Virginia Advanced Research  Computing Services group for their help, with special thanks to Jacalyn Huband.

\end{document}